\documentclass{article}

 \usepackage{authblk}
 \usepackage{graphicx}

 \usepackage{mathtools}
\usepackage{amssymb}

\usepackage{amsthm}

\usepackage{refstyle}
\usepackage{hyperref}

\usepackage{amsmath}
\usepackage{url}
\usepackage{mathrsfs}
\usepackage{amstext}
\usepackage{amssymb}
\usepackage{color}
\usepackage{float}
\usepackage{algorithmic}
\usepackage{longtable}
\usepackage{todonotes}
\usepackage{xspace}

\AtBeginDocument{\providecommand\secref[1]{\ref{sec:#1}}}
\AtBeginDocument{\providecommand\probref[1]{\ref{prob:#1}}}
\AtBeginDocument{\providecommand\propref[1]{\ref{prop:#1}}}
\AtBeginDocument{\providecommand\lemref[1]{\ref{lem:#1}}}
\AtBeginDocument{\providecommand\algref[1]{\ref{alg:#1}}}
\AtBeginDocument{\providecommand\tabref[1]{\ref{tab:#1}}}
\AtBeginDocument{\providecommand\figref[1]{\ref{fig:#1}}}
\AtBeginDocument{\providecommand\defref[1]{\ref{def:#1}}}

\providecommand{\tabularnewline}{\\}

\floatstyle{ruled}
\newfloat{algorithm}{tbp}{loa}
\providecommand{\algorithmname}{Algorithm}
\floatname{algorithm}{\protect\algorithmname}

\newref{rem}{name=Remark~}
\newref{prop}{name=Proposition~}
\newref{thm}{name=Theorem~}
\newref{cor}{name=Corollary~}
\newref{prob}{name=Problem~}
\newref{tab}{name=Table~}
\newref{exa}{name=Example~}
\newref{eq}{name=Eq.~}
\newref{lem}{name=Lemma~}
\newref{def}{name=Definition~}
\newref{fig}{name=Figure~}
\newref{thm}{name=Theorem~}
\newref{sec}{name=Section~}
\newref{subsec}{name=Section~}
\newref{chap}{name=Chapter~}
\newref{alg}{name=Algorithm~}

\usepackage{caption}

\theoremstyle{plain}

\newtheorem{lemma}{Lemma}[section]

\newtheorem{proposition}{Proposition}[section]

\theoremstyle{definition}
\numberwithin{equation}{section}
\newtheorem{definition}{Definition}[section]
\newtheorem{problem}{Problem}[section]

\newtheorem{examplex}{Example}
\newenvironment{example}
  {\pushQED{\qed}\examplex}
  {\popQED\endexamplex}

\global\long\def\pro#1{\text{proj}^{#1}}

\global\long\def\B#1{\{0,1\}^{#1}}

\global\long\def\BoolFunc#1{\mathbb{B}(#1)}

\global\long\def\vI{\mathcal{I}}
\global\long\def\vP{\mathcal{\mathcal{V}}}
\global\long\def\Forward{\mathcal{\text{FORWARD}}}
\global\long\def\Forwardeq{\mathcal{\text{FORWARDEQ}}}
\global\long\def\Backward{\mathcal{\text{BACKWARD}}}
\global\long\def\Smaller{\mathcal{\text{SMALLER}}}
\global\long\def\Smallereq{\mathcal{\text{SMALLEREQ}}}
\def\F{\mathbb{F}_{2}}
\def\Fx{\F[x_{1},\dots,x_{n}]}
\def\gbs{Gr\"obner bases\xspace}
\def\gb{Gr\"obner basis\xspace}
\def\wrt{with respect to\xspace}
\def\Zzn{\mathbb{Z}_{\geq0}^{n}}
\def\Var{\mathrm{Var}}
\def\Comp{\mathrm{Comp}}

\newcommand{\NF}[3]{\mathrm{NF}_{#2}(#3, #1)}
\newcommand{\NFI}[2]{\mathrm{NF}(#2, #1)}
\date{}

\title{Classifier construction in Boolean networks using algebraic
  methods\thanks{Supported by the DFG-funded Cluster of Excellence
    MATH+: Berlin Mathematics Research Center, Project
    AA1-4. Mat\'ias R. Bender was supported by the ERC under
    the European’s Horizon 2020 research and innovation programme
    (grant agreement No 787840).}}
\author[1]{Robert Schwieger}
\author[2]{Mat\'ias R. Bender}
\author[1]{Heike Siebert}
\author[1]{Christian Haase}
\affil[1]{Freie Universit\"at Berlin, Germany}
\affil[2]{Technische Universit\"at Berlin, Germany}

\begin{document}

\maketitle              
\vspace{-2\baselineskip}

\begin{abstract}
We investigate how classifiers for Boolean networks (BNs) can be constructed and modified under constraints. A typical constraint is to observe only states in attractors or even more specifically steady states of BNs. Steady states of BNs are one of the most interesting features for application. Large models can possess many steady states. In the typical scenario motivating this paper we start from a Boolean model with a given classification of the state space into phenotypes defined by high-level readout components. In order to link molecular biomarkers with experimental design, we search for alternative components suitable for the given classification task. This is useful for modelers of regulatory networks for suggesting experiments and measurements based on their models. It can also help to explain causal relations between components and phenotypes. To tackle this problem we need to use the structure of the BN and the constraints. This calls for an algebraic approach. Indeed we demonstrate that this problem can be reformulated into the language of algebraic geometry. While already interesting in itself, this allows us to use \gbs to construct an algorithm for finding such classifiers. We demonstrate the usefulness of this algorithm as a proof of concept on a model with 25 components.

\end{abstract}

{\small
  \textbf{Keywords:} Boolean networks, Algebraic geometry, Gr\"obner bases, Classifiers.
  }
\section{\label{sec:Motivation}Motivation}
For the analysis of large regulatory networks so called \emph{Boolean
networks (BNs)} are used among other modeling frameworks  \cite{samaga2013modeling,albert2014boolean,le2015quantitative}. They have been applied frequently in the past \cite{gonzalez2008logical,sanchez2002segmenting,faure2014discrete,bonzanni2013hard}. In this approach interactions between different components of the
regulatory networks are modeled by logical expressions. Formally,
a Boolean network is simply a Boolean function $f:\B n\rightarrow\B n$, $n\in\mathbb{N}$. This Boolean function contains the information
about the interactions of the components in the network. It is then
translated into a so called \emph{state transition graph (STG)}. There
are several slightly different formalisms for the construction of
the STG of a BN. In all cases, the resulting state transition graph
is a directed graph over the set of vertices $\B n$. The vertices
of the STG are also called \emph{states} in the literature about Boolean
networks. 

Modelers of regulatory networks are frequently \textendash{} if not
to say almost always \textendash{} confronted with uncertainties about
the exact nature of the interactions among the components of the network.
Consequently, in many modeling approaches models may exhibit alternative
behaviors. In so called \emph{asynchronous} Boolean networks for example
each state in the state transition graph can have many potential successor
states (see e.g. \cite{chaouiya2003qualitative}). More fundamentally,
alternative models are constructed and then compared with each other
(see e.g. \cite{samaga2009logic,thobe2017unraveling}).

To validate or to refine such models we need to measure the real world
system and compare the results with the model(s). However, in reality
for networks with many components it is not realistic to be able to
measure all the components. In this scenario there is an additional
step in the above procedure in which the modeler first needs to select
a set of components to be measured which are relevant for the posed
question. This scenario motivates our problem here. How can a modeler
decide which components should be measured? It is clear that the answer
depends on the question posed to the model and on the prior knowledge
or assumptions assumed to be true.

When formalizing this question we are confronted with the task to
find different representations of partially defined Boolean functions.
In the field of \emph{logical analysis of data (LAD)} a very similar
problem is tackled \cite{boros2000implementation,alexe2003coronary,hammer2006logical,chikalov2013logical}.
Here a list of binary vectorized samples needs to be extended to a
Boolean function \textendash{} a so called \emph{theory} (see e.g.
\cite[p. 160]{chikalov2013logical}). In the literature of LAD this
problem is also referred to as the Extension-Problem \cite[p. 161 and p. 170]{chikalov2013logical}.
Here reformulations into linear integer programs are used frequently
\cite{chikalov2013logical}. However, they are more tailored to the
case where the partially defined Boolean functions are defined explicitly
by truth tables. In contrast to the scenario in LAD in our case the
sets are typically assumed to be given implicitly (e.g. by so called
\emph{readout components}).

A common assumption in the field of Boolean modeling is that \emph{attractors}
play an important role. Attractors of BNs are thought
to capture the long term behavior of the modeled regulatory network.
Of special interest among these attractors are \emph{steady states}
(defined by $f(x)=x$ for a BN $f$). Consequently, a typical scenario
is that the modeler assumes to observe only states of the modeled
network which correspond to states belonging to attractors or even
only steady states of the STG. The state space is then often partitioned
by so-called \emph{readout components} into \emph{phenotypes}. 

Our first contribution will be a reformulation of the above problem
into the language of algebraic geometry. For this purpose we focus
on the case of classification into two phenotypes. This is an important
special case. Solutions to the more general case can be obtained by
performing the algorithm iteratively. The two sets of states $A_{1}$
and $A_{2}$ in $\B n$ describing the phenotypes will be defined
by some polynomial equations in the components of the network. This
algebraic reformulation is possible since we can express the Boolean
function $f:\B n\rightarrow\B n$ with polynomials over $\Fx$
\textendash{} the polynomial ring over the finite field of cardinality
two (see \secref{Algebraic-background}). In this way we relate the
problem to a large well-developed theoretical framework. Algebraic
approaches for the construction and analysis of BNs and chemical reaction systems have
been used in the past already successfully (see e.g. \cite{jarrah2007discrete_toric_varieties,laubenbacher_poly_alg_discr_models,millan2012chemical}).
Among other applications they have been applied to the control of BNs  \cite{laubenbacher_identification_control_targets}
and to the inference of BNs from data \cite{veliz2012algebraic,vera2014algebra}.

Our second contribution will be to use this algebraic machinery to
construct a new algorithm to find alternative classifiers. To our
knowledge this is the first algorithm that is able to make use of the
implicit description of the sets that should be classified. For this
algorithm we use \gbs. \gbs are one of the most
important tools in computational algebraic geometry and they have been
applied in innumerous applications, e.g. cryptography
\cite{faugere_algebraic_2003}, statistics \cite{drton_lectures_2009},
robotics \cite{buchberger_applications_1988}, biological dynamical
systems \cite{dickenstein_multistationarity_2019,laubenbacher_reverse_engeneering,laubenbacher_identification_control_targets,laubenbacher_steady_states}.
Specialized algorithms for the computations of \gbs have
been developed for the Boolean case and can be freely accessed
\cite{brickenstein2009polybori}.  They are able to deal with with
systems of Boolean polynomials with up to several hundreds variables
\cite{brickenstein2009polybori} using a specialized data structure (so
called zero-suppressed binary decision diagram (ZDD)
\cite{mishchenko2001introduction}). Such approaches are in many
instances competitive with conventional solvers for the Boolean
satisfiability problem (SAT-solvers) \cite{brickenstein2009polybori}.

Our paper is structured in the following way. We start by giving the
mathematical background used in the subsequent sections in \secref{Algebraic-background}.
In \secref{Algebraic-formalization} we formalize our problem. We
then continue in \secref{Description-of-the-algorithm} to give a
high-level description of the algorithm we developed for this problem.
More details about the used data structures and performance can be
found in \secref{Implementation-and-benchmarking}. As a proof of
concept we investigate in \secref{Biological-example} a BN
of $25$ components modeling cell-fate decision \cite{cell-fate-model}.
We conclude the paper with discussing potential ways to improve the
algorithm.

\section{\label{sec:Algebraic-background}Mathematical background}

In the course of this paper we need some concepts and notation used
in computational algebraic geometry. For our purposes, we will give
all definitions for the field of cardinaliy two denoted by $\mathbb{F}_{2}$
even though they apply to a much more general setting. For a more
extensive and general introduction to algebraic geometry and \gbs we refer to \cite{cox2007ideals}. 

We denote the ring of polynomials in $x_{1},\dots,x_{n}$ over
$\mathbb{F}_{2}$ with $\Fx$. For $n \in \mathbb{N}$, let
$[n] := \{1,\dots,n\}$. Given
$\alpha = (\alpha_1,\dots,\alpha_n) \in \Zzn$, we denote by $x^\alpha$
the monomial $\prod_i x_i^{\alpha_i}$ in $\Fx$. For $f_1,\dots , f_k$
in $\Fx$ we denote with $\langle f_1,\dots , f_k \rangle$ a so-called
\emph{ideal} in $\Fx$ \textendash{} a subset of polynomials which is
closed under addition and multiplication with elements in $\Fx$
\textendash{} generated by these polynomials.
The set of Boolean functions \textendash{} that is the set of
functions from $\mathbb{F}_{2}^{n}$ to $\mathbb{F}_{2}$ \textendash{}
will be denoted by $\BoolFunc n$.  When speaking about Boolean
functions and polynomials in $\Fx$ we need to take into account that
the set of polynomials $\Fx$ does not coincide with the set of Boolean
functions. This is the case since the so-called \emph{field
  polynomials} $x_{1}^{2}-x_{1}$, $\dots$, $x_{n}^{2}-x_{n}$ evaluate
to zero over $\mathbb{F}_{2}^{n}$ \cite{germundsson1991basic}.
Consequently, there is not a one-to-one correspondence between
polynomials and Boolean functions. However, we can say that any two
polynomials whose difference is a sum of field polynomials corresponds
to the same Boolean function (see
e.g. \cite{cheng2009controllability}).  In other words we can identify
the ring of Boolean functions $\BoolFunc n$ with
the quotient ring $\Fx/\langle x_{1}^{2}-x_{1},\dots,x_{n}^{2}-x_{n}\rangle$.  We will
denote both objects with $\BoolFunc n$. A canonical system of
representatives of $\BoolFunc n$ is linearly spanned by the the
square-free monomials in $\Fx$.
Hence, in what follows when we talk about a Boolean function
$f \in \BoolFunc n$ as a polynomial in the variables $x_1,\dots,x_n$
we refer to the unique polynomial in $\Fx$ which involves only
monomials that are square-free and agrees with $f$ as a Boolean
function.

Since we are interested in our application in subsets of $\mathbb{F}_{2}^{n}$,
we need to explain their relationship to the polynomial ring $\Fx$.
This relationship is established using the notion of the vanishing
ideal. Instead of considering a set $B\subseteq\mathbb{F}_{2}^{n}$
we will look at its \emph{vanishing ideal} $\vI(B)$ in $\BoolFunc n$.
The vanishing ideal of $B$ consists of all Boolean functions which
evaluate to zero on $B$. Conversely, for an ideal $\vI$ in $\BoolFunc n$
we denote with $\vP(\vI)$ the set of points in $\mathbb{F}_{2}^{n}$
for which every Boolean function in $\vI$ evaluates to zero. Due
to the Boolean Nullstellensatz (see \cite{sato2011boolean,finite_fields_qunatifier_elimination})
there is an easy relation between a set $B\subseteq\mathbb{F}_{2}^{n}$
and its vanishing ideal $\vI(B)$: For an ideal $\vI$ in $\BoolFunc n$
such that $\vP(\vI)\not=\emptyset$ and for any polynomial $h \in\BoolFunc n$
it holds
\[
h \in\vI\Leftrightarrow\forall v\in\vP(\vI):h(v)=0.
\]

In this paper, we will consider Boolean functions whose domain is restricticted to certain states (e.g. attractors or steady states). Hence, there are different Boolean functions that behave in the same way when we restrict their domain.
\begin{example}
  Consider the set $B :=
  \{000, 110, 101, 011\}$. Consider the Boolean function
  $f := x_1$ and $g := x_2 + x_3$. Both Boolean functions are
  different, i.e., $f(1,1,1) = 1$ and $g(1,1,1) = 0$, but they agree
  over $B$.
  $$
  \begin{array}{c |  c | c  | c | c}
    & 000 & 110 & 101 & 011 \\ \hline
    f & 0 & 1 & 1 & 0 \\
    g & 0 & 1 & 1 & 0 
  \end{array}.
  $$
  Note that $\vI(B) = \langle x_1 + x_2 + x_3 \rangle$, that is, the
  ideal $\vI(B)$ is generated by the Boolean function
  $x_1 + x_2 + x_3$ since it is the unique Boolean function vanishing only on $B$.
\end{example}
Given a set $B$, we write $\BoolFunc n/\vI(B)$ to refer to the set of
all the different Boolean functions on $B$. As we saw in the previous
example, different Boolean functions agree on $B$. Hence, we will be
interested in how to obtain certain representatives of the Boolean
function in $\BoolFunc n/\vI(B)$ algorithmically. In our application,
the set $\BoolFunc n/\vI(B)$ will become the set of all possible
classifiers we can construct that differ on $B$.
To obtain specific representatives of a Boolean function in $\BoolFunc
n/\vI(B)$ we will use \gbs.
A \gb of an ideal is a set of generators of the ideal with some extra
properties related to \emph{monomial orderings}.  A monomial ordering is a total ordering on the set of monomials in $\Fx$ satisfying some additional properties to ensure the compatibility with the algebraic operations in $\Fx$ (see \cite[p. 69]{cox2007ideals} for details). 
 
For any polynomial in $p\in\Fx$ and monomial ordering $\prec$, we denote the \emph{initial monomial} of $p$ by $in_\prec(p)$, that is the largest monomial appearing
in $p$ with respect to $\prec$.
We are interested in specific orderings \textendash{} the
lexicographical orderings \textendash{} on these monomials. As we will see, the usage of lexicographical orderings in the context of our application will allow us to look for classifiers which are optimal in a certain sense.
\begin{definition}[{\cite[p. 70]{cox2007ideals}}]
  \label{def:lexicographical-orderings}
  Let
  $\alpha=\begin{pmatrix}\alpha_{1} & \dots & \alpha_{n}\end{pmatrix}$
  and
  $\beta=\begin{pmatrix}\beta_{1} & \dots & \beta_{n}\end{pmatrix}$ be
  two elements in $\Zzn$.
  Given a permutation $\sigma$ of $\{1,\dots,n\}$, we say
  $x^\alpha \succ_{lex(\sigma)} x^\beta$ if there is
  $k \in [n]$ such that
  $$
  (\forall i < k: \, \alpha_{\sigma(i)} = \beta_{\sigma(i)})
  \text{ and }  \alpha_{\sigma(k)} > \beta_{\sigma(k)}.
  $$
\end{definition}

\begin{definition}[{\cite[p. 1]{sturmfels1996grobner}}]
\label{def:initial_ideal}Let $\prec$ be any monomial ordering. For
an ideal $\vI\subseteq\Fx$ we define its \emph{initial
ideal} as the ideal 
\[
in_{\prec}(\vI):=\langle in_{\prec}(f)|f\in\vI\rangle.
\]
A finite subset $G\subseteq\vI$ is a \emph{\gb}
for $\vI$ with respect to $\prec$ if $in_{\prec}(\vI)$ is generated
by $\{in_{\prec}(g)|g\in G\}$. If no element of the \gb
$G$ is redundant, then $G$ is \emph{minimal}.
It is called \emph{reduced} if for any
two distinct elements $g,g'\in G$ no monomial in $g'$ is divisible by
$in_{\prec}(g)$.
Given an ideal and a monomial ordering, there is a unique minimal
reduced \gb involving only monic polynomials; we denote it
by $G_{\prec}(\vI)$. Every monomial not lying in $in_{\prec}(\vI)$ is
called \emph{standard monomial}.
\end{definition}

We can extend the monomial orderings to partial orderings of
polynomials on $\Fx$.
Consider polynomials $f, g \in \Fx$ and a monomial ordering
$\succ$. We say that $f \succ g$ if $in_\prec(f) \succ in_\prec(g)$ or
$f - in_\prec(f) \succ g - in_\prec(g)$.
The division algorithm rewrites every polynomial $f \in \Fx$ modulo
$\vI$ uniquely as a linear combination of these standard
monomials \cite[Ch. 2]{cox2007ideals}. The result of this algorithm is
called Normal form.
For the convenience of the reader, we state this in the
following lemma and definition.

\begin{lemma}[Normal form]
  Given a monomial ordering $\succ$, $f \in \BoolFunc n$ and an
  ideal $\vI$ there is a unique $g \in \BoolFunc n$ such that $f$ and
  $g$ represent the same Boolean function in $\BoolFunc n / \vI$ and
  $g$ is minimal with respect to the ordering $\succ$ among all the
  Boolean functions equivalent to $f$ in $\BoolFunc n / \vI$. We call $g$ the normal form of $f$ modulo $I$ denoted by
  $\NF{f}{\vI}{\succ}$.
\end{lemma}

\begin{example}[Cont.]
  Consider the permutation $\sigma$ of $\{1,2,3\}$ such that
  $\sigma(i) := 4 - i$. Then,
  $\NF{f}{\vI}{\succ_{\sigma}} = \NF{g}{\vI}{\succ_{\sigma}} =
  x_1$. Hence, if we choose a ``good'' monomial ordering, we can
  get simpler Boolean functions involving less variables.
\end{example}

\section{Algebraic formalization}
\label{sec:Algebraic-formalization}

As discussed in \secref{Motivation}, we start with the assumption
that we are given a set of Boolean vectors $B$ in $\mathbb{F}_{2}^{n}$
representing the observable states. In our applications these states
are typically attractors or steady states of a BN. We also assume that
our set $B$ is partitioned into a set of phenotypes, i.e. $B=A_{1}\cup\cdots\cup A_{k}$.
Our goal is then to find the components that allow us to decide for
a vector in $B$ to which set $A_{i}$, $i\in[k]$, it belongs.
For a set of indices $I\subseteq[n]$ and $x \in \mathbb{F}_{2}^{n}$, let us denote with $\pro I(x)$ the projection of $x$ onto the components $I$.  Our problem could be formalized in the following way.

\begin{problem}[State-Discrimination-Problem]
\label{prob:problem-statement}For a given partition of non-empty
sets $A_{1},\dots,A_{k}$ ($k\geq2$) of states $B\subseteq\B n$, find
the sets of components $\emptyset\not=I\subseteq[n]$ such that $\pro I(A_{1}),\dots,\pro I(A_{k})$
forms a partition of $\pro I(B)$. 
\end{problem}

Clearly, since the sets $A_{1},\dots,A_{k}$ form a partition of $B$, we can decide for each state $x$ in $B$ to which set $A_i$ it belongs. If $I\subseteq[n]$ is a solution to \probref{problem-statement}, this decision can be only based on $\pro I(x)$ as $\pro I(A_{1}),\dots,\pro I(A_{k})$ form a partition of $\pro I(B)$. As we discussed in \secref{Motivation}, \probref{problem-statement}
is equivalent to Extension-Problem \cite[p. 161 and p. 170]{chikalov2013logical}.
However, in our case the sets in \probref{problem-statement} are
typically given implicitly. That is, we are given already some information
about the structure of the sets in the above problem. This calls for
an algebraic approach. We consider the case in \probref{problem-statement}
where $k$ equals two. This is an important special case since many
classification problems consist of two sets (e.g. healthy and sick).
Furthermore, solutions to the more general case can be obtained by
considering iteratively the binary case (see also the case study in
\secref{Biological-example}).

Let $\vI(B)\subseteq\Fx$ be the vanishing
ideal of a set $B\subseteq\mathbb{F}_{2}^{n}$. Let $f:\mathbb{F}_{2}^{n}\rightarrow\mathbb{F}_{2}$
be a Boolean function which can be identified with an element in $\BoolFunc n:=\Fx\big/\langle x_{1}^{2}-x_{1},\dots,x_{n}^{2}-x_{n}\rangle$.
We want to find representatives of $f$ in $\BoolFunc n\big/\vI(B)$
which depend on a minimal set of variables with respect to set inclusion
or cardinality. We express this in the following form:

\begin{problem}
  \label{prob:version1}
  For $f \in \BoolFunc n$ and
  $B\subseteq\mathbb{F}_{2}^{n}$, find the representatives of $f$ in
  $\BoolFunc n\big/\vI(B)$ which depend on a set of variables
  satisfying some minimality criterion.
\end{problem}

It is clear that \probref{version1} is equivalent to
\probref{problem-statement} for the case $k=2$ since due to the
Boolean Strong Nullstellensatz (see \cite{sato2011boolean}) a Boolean
function $f$ is zero on $B$ if and only if $f$ is in
$\vI(B)$. Therefore, all Boolean functions which are in the same
residue class as $f$ agree with it as Boolean functions on
$B$ and vice versa. The sets of variables each representative depends
on are the solutions to \probref{version1}. The representatives are
then the classifiers.

Here we will focus on solutions of \probref{version1} which are minimal
with respect to set inclusion or cardinality. However, also other
optimality criteria are imaginable. For example one could introduce
some weights for the components. Let us illustrate \probref{version1}
with a small example.

\begin{example}
Consider the set $B=\{000,111,011,101\}\subset\mathbb{F}_{2}^{3}$.
Then $\vI(B)\subseteq\BoolFunc 3$ is given by $\langle\overbrace{x_{1}x_{2}+x_{1}+x_{2}+x_{3}}^{=:f}\rangle$ since
$f$ is the unique Boolean function that is zero on $B$ and one on it complement.
Let $\varphi(x)=x_{1}x_{2}x_{3}$. It is easy to check that for example
$x_{1}x_{3}+x_{2}x_{3}+x_{3}$ and $x_{1}x_{2}$ are different representatives
of $\varphi$ in $\BoolFunc 3\big/\vI(B)$. The representative $x_{1}x_{2}$
depends only on two variables while the other two representatives
depend on three. 
\end{example}

We can obtain a minimal representative of $f$ in \probref{version1} by
computing $\NF{f}{\vI}{\prec}$ for a suitable lexicographical ordering
$\prec$.

\begin{proposition}
  \label{prop:lex-orderings-are-useful}
  Given a set of points $B \subset \F^n$ and a Boolean function
  $f \in \BoolFunc n$ assume, with no loss of generality, that there
  is an equivalent Boolean function $g \in \Fx$ modulo $\vI(B)$
  involving only $x_k,\dots,x_n$.  Consider a permutation $\sigma$ of
  $\{1,\dots,n\}$ such that $\sigma(\{k,\dots,n\}) =
  \{k,\dots,n\}$. Then, the only variables appearing in
  $\NF{f}{I}{\prec_{lex(\sigma)}} = \NF{g}{I}{\prec_{lex(\sigma)}}$
  are the ones in $\{x_k,\dots,x_n\}$. In particular, if there is no
  Boolean function equivalent to $f$ modulo $\vI(B)$ involving a
  proper subset of $\{x_k,\dots,x_n\}$, then
  $\NF{f}{I}{\prec_{lex(\sigma)}} = \NF{g}{I}{\prec_{lex(\sigma)}}$
  involves all the variables in $\{x_k,\dots,x_n\}$.
\end{proposition}

\begin{proof}
  The proof follows from the minimality of
  $\NF{f}{I}{\prec_{lex(\sigma)}}$ \wrt $\prec_{lex(\sigma)}$. Note
  that, because of the lexicographical ordering $\prec_{lex(\sigma)}$, any
  Boolean function equivalent to $f$ modulo $\vI(B)$ involving 
  variables in $\{x_1,\dots,x_{k-1}\}$ will be bigger than $g$, so it
  cannot be minimal.
\end{proof}

\section{\label{sec:Description-of-the-algorithm}Description of the Algorithm}

Clearly we could use \propref{lex-orderings-are-useful} to obtain
an algorithm that finds the minimal representatives of $\varphi$
in $\BoolFunc n/\vI(B)$ by iterating over all lexicographical orderings
in $\Fx$. However, this naive approach
has several drawbacks:

\begin{enumerate}
\item The number of orderings over $\BoolFunc n$ is growing rapidly with
$n$ since there are $n!$ many lexicographical orderings over $\BoolFunc n$
to check.
\item We do not obtain for every lexicographical ordering a minimal representative.
Excluding some of these orderings ``simultaneously'' could be very
beneficial.
\item Different monomial orderings can induce the same \gbs. Consequently, the normal form leads to the same
  representative.
\item Normal forms with different monomial orderings can result in the
  same representative. If we detect such cases we avoid unnecessary
  computations.
\end{enumerate}

We describe now an algorithm addressing the first two points. Recall
that for a monomial ordering $\succ$ and 
$f \in \BoolFunc n$, we use the notation $\NF{f}{\vI}{\succ}$ to
denote the normal form of $f$ in $\BoolFunc n / \vI$ \wrt
$\succ$. When $\vI$ is clear from the context, we write
$\NFI{f}{\succ}$.  We denote with $\varphi$ any
representative of the indicator function of $A$ in $\BoolFunc n\big/\vI(B)$. Let 
$\Var(\varphi)$ be the variables occurring in $\varphi$ and 
$\Comp(\varphi)$ be its complement in $\{x_{1},\dots,x_{n}\}$. 
Instead of iterating through the orderings on $\BoolFunc n$, we
consider candidate sets in the power set of $\{x_{1},\dots,x_{n}\}$,
denoted by $\mathscr{P}(x_{1},\dots,x_{n})$ (i.e. we initialize the
family of candidate sets $P$ with
$P\leftarrow\mathscr{P}(x_{1},\dots,x_{n})$).  We want to find the
sets $A$ in the family of candidate sets $P$ for which the equality
$A=\Var(\varphi)$ holds for some minimal solution $\varphi$ to
\probref{version1}, that is, involving the minimal amount of
variables. For each candidate set $A$ involving $k$ variables 
we pick a lexicographical ordering $\succ$ for which it holds
$A^{c}\succ A$, i.e. for every variable $x_i \in A^{c}$ and
$x_j \in A$ it holds $x_i \succ x_j$, where
$A^c := \{x_1,\dots,x_n\} \setminus A$.  This approach is sufficient
to find the minimal solutions as we will argue below. This addresses
the first point above since there are $2^{n}$ candidate sets to
consider while there are $n!$ many orderings.\footnote{Note that
  $\lim_{n \rightarrow \infty} \frac{n!}{2^n} = \infty$, so it is more
  efficient to iterate through $2^{n}$ candidate sets than through
  $n!$ orderings.}

\subsection{Excluding candidate sets}

To address the second point we will exclude after each reduction step
a family of candidate sets. 
If, for an ordering $\succ$, we computed a representative $\NFI{\varphi}{\succ}$
we can, independently of the minimality of $\NFI{\varphi}{\succ}$, exclude
some sets in $P$. To do so, we define for any set $A\subseteq\{x_{1},\dots,x_{n}\}$
 the following families of sets:

\begin{align*}
\Forward(A) & :=\big\{ B\subseteq\{x_{1},\dots,x_{n}\}|A\subset B\big\},\\
\text{\ensuremath{\Forwardeq}}(A) & :=\big\{ B\subseteq\{x_{1},\dots,x_{n}\}|A\subseteq B\big\},\\
\text{\ensuremath{\Backward}}(A) & :=\big\{ B\subseteq\{x_{1},\dots,x_{n}\}|B\subseteq A\big\},\\
\Smaller(A,\succ) & :=\big\{ x\in\{x_{1},\dots,x_{n}\}|\exists y\in A:y\succ x\big\},\\
\Smallereq(A,\succ) & :=\big\{ x\in\{x_{1},\dots,x_{n}\}|\exists y\in A:y\succeq x\big\}.
\end{align*}

It is clear that, if we obtain in a reduction step a representative
$\phi=\NFI{\varphi}{\succ}$, we can exclude the sets in
$\Forward(\Var(\phi))$ from the candidate sets $P$. But, as we see in the following lemma, we can exclude even
more candidate sets.

\begin{lemma}
\label{lem:A-is-not-solution-2}Let $\succ$ be a lexicographical
ordering and let $\phi=\NFI{\varphi}{\succ}$ be the corresponding normal form of $\varphi$. Then none of the sets $A\subseteq\Smaller(\Var(\phi),\succ)$
can belong to a minimal solution to \probref{version1}. 
\end{lemma}

\begin{proof}
  Assume the contrary, that is there is a minimal solution $\psi$
  with $\Var(\psi) \subseteq\Smaller(\Var(\phi),\succ)$. It follows that, by the definition of lexicographical orderings, there is
  at least one $y\in \Var(\phi)$ with $y\succ \Var(\psi)$. Consequently,
  $\psi$ is smaller than $\phi$ \wrt $\succ$ which cannot happen by the
  definition of the normal form.
\end{proof}

If we also take the structure of the polynomials into account, we can
improve \lemref{A-is-not-solution-2} further. For this purpose, we
look at the initial monomial of $\phi=\NFI{\varphi}{\succ}$ with
respect to $\succ$, $M := in_\succ(\phi)$. We consider the sets $\Var(M)$ and $\Comp(M)$.
Given a variable $x_{i}\in\{x_{1},\dots,x_{n}\}$ and a subset
$S\subseteq\{x_{1},\dots,x_{i}\}$,
let $S_{\succ x_{i}}$ be the set of variables in $S$ bigger than
$x_i$, i.e.
\[
S_{\succ x_{i}} := \big\{ x_j \in S| x_j \succ x_{i}\big\}.
\]

\begin{lemma}
  \label{lem:A-is-not-solution-3}
  Consider a lexicographical ordering $\succ$. Let
  $\phi=\NFI{\varphi}{\succ}$ and $M=in_\succ(\text{\ensuremath{\phi}})$.
  If $x_{i}\in \Var(M)$, then any set
  $S\subseteq\Smallereq(\Var(\phi),\succ)$ with $x_{i}\not\in S$ and
  $S\cap Comp_{\succ x_{i}}(M)=\emptyset$ cannot belong to a minimal
  solution to \probref{version1}.
\end{lemma}

\begin{proof}
  Note that $\prod_{j\in S}x_{j} \prec M$ as $M$ involves $x_i$
  but $\prod_{j\in S}x_{j}$ only involves variables smaller than
  $x_i$. Then, the proof is analogous to \lemref{A-is-not-solution-2}
  using the fact that any minimal solution involving only variables in
  $S$ has monomials smaller or equal than
  $\prod_{j\in S}x_{j} \prec M$. Hence, $\phi$ is not minimal \wrt $\prec$.
\end{proof}
In particular, \lemref{A-is-not-solution-3} entails the case where
the initial monomial of $\NFI{\varphi}{\succ}$ is a product of all variables
occurring in $\NFI{\varphi}{\succ}$. In this case, for every subset $S\subseteq \Var(\phi)\subseteq\Smallereq(\Var(\phi),\succ)$
it holds
$S\cap\Comp(M) = S\cap\Comp(\phi) = \emptyset$.
Therefore,
according to \lemref{A-is-not-solution-3} $\NFI{\varphi}{\succ}$ is minimal. 

For a lexicographical ordering $\succ$ and a normal form $\phi=\NFI{\varphi}{\succ}$
we can, using \lemref{A-is-not-solution-3}, exclude the families
of sets in (\ref{eq:Si}) from the set of candidates $P$. 
\begin{align}
  & \Backward(S_{i})\text{ with }x_{i}\in \Var(in_\succ(\phi))\text{ and}\label{eq:Si}\\
 & S_{i}:=\Smallereq(\Var(\phi),\succ)\backslash\big(\{x_{i}\}\cup Comp_{\succ x_{i}}(in_\succ(\phi))\big)\nonumber 
\end{align}
We illustrate this fact with a small example:
\begin{example}
  Consider a lexicographical ordering $\succ$ with
  $x_{4}\succ\cdots\succ x_{1}$ and a normal form
  $\phi=\NFI{\varphi}{\succ}$ with initial monomial
  $x_{4}x_{2}$. Then, we can exclude from $P$ the sets in
  $\Backward(\{x_{3},x_{2},x_{1}\})$ and $\Backward(\{x_{4},x_{1}\})$.
\end{example}

Note that if we consider, instead of lexicographical orderings, graded
monomial orderings, then we obtain the following version of
\lemref{A-is-not-solution-3}.  This is useful to lower bound the
number of variables in a minimal solution. Also, it could be useful
when considering different optimality criteria.
\begin{lemma}
  Let $\succ$ be a graded monomial ordering \cite[Ch.
  8.4]{cox2007ideals}. Then, the total degree $d$ of
  $\NFI{\varphi}{\succ}$ is smaller or equal to the number of
  variables involved in any minimal representation of $\varphi$.
\end{lemma}

\begin{proof}
  Assume that $\varphi$ has a representation involving less than $d$
  variables. Then, this representation has to have degree less than
  $d$ (because every monomial is square-free). Hence, we get a
  contradiction because $\NFI{\varphi}{\succ}$ is not minimal.
\end{proof}

We can now use the results above to construct
\algref{compute_min_repr}.  In each step of our algorithm we choose a
candidate set $A$ of $P$ and an ordering $\succ$ satisfying
$A^{c}\succ A$. Then we compute the reduction of $\varphi$ with
respect to $\succ$ with the corresponding \gb. Let us call the result
$\phi$. After each reduction in \algref{compute_min_repr} we exclude
from $P$ the sets that we already checked and the sets we can exclude
with the results above.  That is, we can exclude from $P$ the
candidate sets $\Forwardeq(\Var(\phi))$ (Line $9$ in
\algref{compute_min_repr}) and according to
\lemref{A-is-not-solution-3} the family of sets
$\Backward(S_{i})\text{ with }x_{i}\in \Var(in_\prec(\phi))$ where $S_{i}$
is defined according to (\ref{eq:Si}). The algorithm keeps doing this
until the set of candidate sets is empty. To be able to return the
solutions we keep simultaneously track of the set of
potential solutions denoted by $S$. Initially this set equals $P$.  But since we
subtract from $S$ not the set $\Forwardeq(\Var(\phi))$ but
$\Forward(\Var(\phi))$ we keep some of the sets that we checked
already in $S$. This guarantees that $S$ will contain all solutions
when $P$ is empty.

\begin{algorithm}

\begin{algorithmic}[1]
\STATE $P\gets \mathscr{P}(\{x_1, \dots , x_n\})$
\STATE $S\gets \mathscr{P}(\{x_1, \dots , x_n\})$
\WHILE{$P \not = \emptyset$}
\STATE $A \gets \text{any set in }P$
\STATE $\succ \gets \text{any lexicographical ordering satisfying }\Comp(A)\succ A$
\STATE $\varphi \gets \NFI{\varphi}{\succ}$
\STATE $V \gets  \Var(\varphi)$

\STATE $P \gets  P - \Forwardeq(V)$
\STATE $S \gets  S - \Forward(V)$
\FORALL{$x_i$ in $\Var(in_\prec(\varphi))$}
\STATE $S_i \gets  \text{Compute $S_i$ according to Eq.(\ref{eq:Si})}$
\STATE $P \gets  P - \Backward(S_i)$
\STATE $S \gets  S - \Backward(S_i)$
\ENDFOR

\ENDWHILE

\RETURN $S$
\end{algorithmic}

\caption{\label{alg:compute_min_repr}compute\_solutions($\varphi$, $\{x_1, \dots, x_n\}$, $\vI$)}
\end{algorithm}

\section{\label{sec:Implementation-and-benchmarking}Implementation and benchmarking}

When implementing \algref{compute_min_repr} the main difficulties
we face is an effective handling of the candidate sets. In each step
in the loop in \algref{compute_min_repr} we need to pick a new set
$A$ from the family of candidate sets $P$. Selecting a candidate
set from $P$ is not a trivial task since it structure can become very
entangled. The subtraction of the
sets $\Forward(\cdot)$ and $\Backward(\cdot)$ from $P$ can make the
structure of the candidate sets very complicated. In practice this
a very time-consuming part of the algorithm. To tackle this problem
we use a specialized data structure \textendash{} so-called Zero-suppressed
decision diagram (ZDDs) \cite{minato1993zero,mishchenko2001introduction}
\textendash{} to represent $P$. ZDDs are a type of Binary Decision Diagrams (BDDs). A binary decision diagram represents a Boolean function or a family of sets as a rooted
directed acyclic graph. Specific reduction rules are used to obtain a compact, memory efficient representation. ZDDs can therefore effectively store families
of sets. Furthermore, set operations can be computed directly on ZDDs. This makes
them an ideal tool for many combinatorial problems \cite{minato1993zero}.
 We refer to the literature for a more detailed introduction to ZDDs \cite{minato1993zero,mishchenko2001introduction}.

 Our implementation in Python can be found at \url{https://git.io/Jfmuc}.
 For the \gb
calculations as well as the ZDDs we used libraries from PolyBoRi (see
\cite{brickenstein2009polybori}). 
The computation time for the network with 25 components
considered in the case study in \secref{Biological-example} was around
10 seconds on a personal computer with an intel core i5 vPro processor. Other networks we created for test purposes resulted in
similar results (around 30 seconds for example1.py in the repository).
However, computation time depends highly on the structure of the network
and not only on its size. For a network even larger (example2.py in the repository with 38 components) computations took around two seconds while computations for a slightly different network of the same size (see example3.py in the repository) took around a minute. For a similar network (example4.py in the repository) we aborted the computation after one hour.
  In general, the complexity of computing \gbs is highly influenced by algebraic properties such as the regularity of the vanishing ideal of the set we restrict our classifiers to. If the shapes of the sets in our algorithm are more regular (e.g. some components are fixed to zero or one) the number of candidate sets is reduced much faster by the algorithm. Similar, computations seem to be also often faster for networks with fewer regulatory links.

\section{\label{sec:Biological-example}Case study}

Let us consider the Boolean model constructed in \cite{cell-fate-model}
modeling cell-fate decision. These models can be used to identify
how and under which conditions the cell chooses between different
types of cell deaths and survival. The complete model can be found
in the BioModels database with the reference MODEL0912180000. It consists
of $25$ components. The corresponding Boolean function is depicted
in \tabref{Calzone-example}. While in \cite{cell-fate-model} the
authors use a reduced model (see also \cite{reduction_paper}) to
make their analysis more tractable, we can and do work with the complete
model here.

The Boolean network depicted in \tabref{Calzone-example} models the
effect of cytokines such as TNF and FASL on cell death. In the Boolean
model they correspond to input components. These cytokines can trigger
cell death by apoptosis or necrosis (referred to as non-apoptotic
cell death abbreviated by NonACD). Under different cellular conditions
they lead to the activation of pro-survival signaling pathway(s).
Consequently, the model distinguishes three phenotypes: Apoptosis,
NonACD and Survival. Three corresponding signaling pathways are unified
in their model. Finally, specific read-out components for the three phenotypes were
defined. The activation of CASP3 is considered
a marker for apoptosis. When MPT occurs and the level of ATP drops
the cell enters non-apoptotic cell death. If Nf$\kappa$B is activated
cells survive \cite[p. 4]{cell-fate-model}. This leads to the three
classifiers in the model depicted in \tabref{Classifier}. Each classifier
tells us to which cell fate (apoptosis, NonACD, survival) a state
belongs. 

We are interested in alternative classifiers on the set of attractors
of the Boolean network. Let us denote the union of these attractors\footnote{An attractor of a Boolean network is a terminal strongly connected component
of the corresponding state transition graph.} with $B$ (in agreement with the notation in \probref{version1}).
In this case all attractors are steady states (see \cite[p. 4]{cell-fate-model}
for details). For illustrating our results we computed the steady
states of the network using GINsim \cite{GinSim} (see \figref{-steady-states-of-network}).
But this is not necessary for our calculations here. However, we can
see that the classifiers given in \cite{cell-fate-model} indeed
result in disjoint sets of phenotypes.

Since the Boolean network in \tabref{Calzone-example} possesses only
steady states as attractors we can represent the ideal $\vI(B)$ in
\probref{version1} as $\langle f_{1}(x)+x_{1},\dots,f_{n}(x)+x_{n}\rangle$
where $f$ is the Boolean function depicted in \tabref{Calzone-example}.

Next, we computed for each of the classifiers alternative
representations.  In Table~\ref{tab:classifier-celldeath}, we present
the nine different minimal representations on $B$ of the classifier
for NonACD. Among these options there are three ways how
to construct a classifier based on one component (that is ATP, MPT or
ROS). Also
interestingly none of the components in the Boolean network is
strictly necessary for the classification of the
phenotypes. Consequently, there are potentially very different
biological markers in the underlying modeled regulatory
network. Despite this, there are some restrictions on the construction
of the classifier, e.g., if we want to use the component Cytc,
MOMP or SMAC we need to use also the component labeled as apoptosome.
In total, the components useful for the classification of NonACD are
ATP, CASP3, Cytc, MOMP, apoptosome, MPT, ROS and SMAC. The remaining
$17$ components are redundant for this purpose.

We obtain similar results for the other two classifiers. For apoptosis
we found $17$ alternative classifiers depicted in Table~\ref{tab:apoptosis}
involving the nine components (ATP, BAX, CASP8, Cytc, MOMP, SMAC,
MPT, ROS, CASP3 and apoptosome). For the classifier for survival of
the cell depicted in Table~\ref{tab:survival} we found much more
alternative classifiers ($84$ alternative classifiers). Most classifiers
depend on four components. But we can observe that each of the components
IKK, BCL2, NFKB1, RIP1ub, XIAP, cFLIP can be used for classification. Computations for each of the three classifiers took around $10$-$30$
seconds on a personal computer with an intel core i5 vPro processor
in each case.

{\small
\begin{table}
\begin{tabular}{ l | p{0.8\textwidth} }
\hline 
Component   & Update function \tabularnewline
\hline 
\hline 
$ ATP $  & $1+MPT$ \tabularnewline
\hline 
$ BAX $  & $CASP8  \cdot  (1+BCL2)$ \tabularnewline
\hline 
$ BCL2 $  & $NFKB1$ \tabularnewline
\hline
$ CASP3 $  & $(1+XIAP)  \cdot  apoptosome$ \tabularnewline
\hline 
$ CASP8 $  & $((1+DISCTNF) \cdot (1+DISCFAS) \cdot CASP3 \cdot (1+cFLIP) + (1+DISCTNF) \cdot DISCFAS \cdot (1+cFLIP)
+(1+DISCTNF) \cdot (1+DISCFAS) \cdot CASP3 \cdot (1+cFLIP) + (1+DISCTNF) \cdot DISCFAS \cdot (1+cFLIP)) \cdot DISCTNF \cdot (1+cFLIP)
+ ((1+DISCTNF) \cdot (1+DISCFAS) \cdot CASP3 \cdot (1+cFLIP) + (1+DISCTNF) \cdot DISCFAS \cdot (1+cFLIP)
+(1+DISCTNF) \cdot (1+DISCFAS) \cdot CASP3 \cdot (1+cFLIP) + (1+DISCTNF) \cdot DISCFAS \cdot (1+cFLIP))
+ DISCTNF \cdot (1+cFLIP)$ \tabularnewline
\hline 
$ Cytc $  & $MOMP$ \tabularnewline
\hline 
$ DISCFAS $  & $FASL \cdot FADD$ \tabularnewline
\hline 
$ DISCTNF $  & $TNFR \cdot FADD$ \tabularnewline
\hline 
$ FADD $  & $FADD$ \tabularnewline
\hline 
$ FASL $  & $FASL$ \tabularnewline
\hline 
$ IKK $  & $RIP1ub$ \tabularnewline
\hline 
$ MOMP $  & $((1+BAX) \cdot MPT) \cdot BAX + ((1+BAX) \cdot MPT) + BAX$ \tabularnewline
\hline 
$ MPT $  & $(1+BCL2) \cdot ROS$ \tabularnewline
\hline 
$ NFKB1 $  & $IKK \cdot (1+CASP3)$ \tabularnewline
\hline 
$ NonACD $  & $1+ATP$ \tabularnewline
\hline 
$ RIP1 $  & $(1+TNFR) \cdot DISCFAS \cdot (1+CASP8) \cdot TNFR \cdot (1+CASP8) + (1+TNFR) \cdot DISCFAS \cdot (1+CASP8)
+ TNFR \cdot (1+CASP8)$ \tabularnewline
\hline 
$ RIP1K $  & $RIP1$ \tabularnewline
\hline 
$ RIP1ub $  & $RIP1 \cdot cIAP$ \tabularnewline
\hline 
$ ROS $  & $(1+RIP1K) \cdot MPT \cdot NFKB1  \cdot  RIP1K \cdot (1+NFKB1) + RIP1K \cdot (1+NFKB1)
+ (1+RIP1K) \cdot MPT \cdot NFKB1$ \tabularnewline
\hline 
$ SMAC $  & $MOMP$ \tabularnewline
\hline 
$ TNF $  & $TNF$ \tabularnewline
\hline 
$ TNFR $  & $TNF$ \tabularnewline
\hline 
$ XIAP $  & $(1+SMAC) \cdot NFKB1$ \tabularnewline
\hline 
$ apoptosome $  & $ATP \cdot Cytc \cdot (1+XIAP)$ \tabularnewline
\hline 
$ cFLIP $  & $NFKB1$ \tabularnewline
\hline 
$ cIAP $  & $(1+NFKB1) \cdot (1+SMAC) \cdot cIAP  \cdot  NFKB1 \cdot (1+SMAC) + (1+NFKB1) \cdot (1+SMAC) \cdot cIAP
+ NFKB1 \cdot (1+SMAC)$ \tabularnewline
\hline
\end{tabular}

\caption{\label{tab:Calzone-example}Boolean network with $25$ components
given in \cite{cell-fate-model}.}
\end{table}}

\begin{table}
\begin{tabular}{ l | p{0.6\textwidth} }
\hline 
Bio. interpretation of classifier & Classifier\tabularnewline
\hline 
\hline 
Survival & $NFKB1$\tabularnewline
\hline 
Apoptosis & $CASP3$\tabularnewline
\hline 
NonACD & $1+ATP$\tabularnewline
\hline 
\end{tabular}

\caption{\label{tab:Classifier}Classifiers for the Boolean network depicted
in \tabref{Calzone-example}.}
\end{table}

\begin{figure}
\centering
  \includegraphics[width=0.8\textwidth]{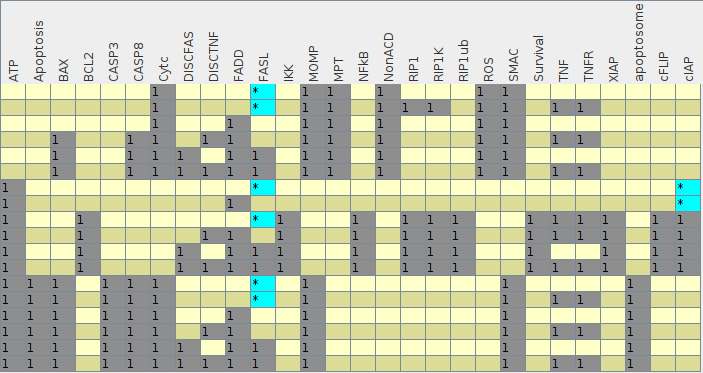}

\caption{\label{fig:-steady-states-of-network}$27$ steady states of the complete
BN computed with GINsim \cite{GinSim}. Components marked with * can either be set to zero or one. Steady states are grouped
into phenotypes. Six steady states are not corresponding to any phenotype.}
\end{figure}

\section{Possible further improvements}

There is still some room for further improvement of the above algorithm.
We address the third point in the beginning of \secref{Description-of-the-algorithm}. We can represent the lexicographical
orderings on $\BoolFunc n$ using weight vectors $w\in\mathbb{N}^{n}$. More precisely,
let $\succ$ be any monomial ordering in $\Fx$
and $w\in\mathbb{N}^{n}$ any weight vector. Then we define $\succ_{w}$
as follows: for two monomials $x^{\alpha}$ and $x^{\beta}$, $\alpha,\beta\in\mathbb{N}^{n}$
we set
\[
x^{\alpha}\succ_{w}x^{\beta} \Leftrightarrow w\cdot\alpha>w\cdot\beta\text{ or }(w\cdot\alpha=w\cdot\beta\text{ and }x^\alpha\succ x^\beta)\cdot
\]
According to \cite[Prop 1.11]{sturmfels1996grobner} for every monomial
ordering $\succ$ and for every ideal in $\Fx$
there exists a non-negative integer vector $w\in\mathbb{N}^{n}$ s.t.
$in_{w}(\vI)=in_{\succ}(\vI)$ where $in_{w}(\vI)$ is the ideal generated by the initial forms $in_{w}(f), f \in \vI$ \textendash that is, the sum of monomials $x^\alpha$ in $f$ which are maximal with respect to the inner product $\alpha \cdot w$. We say also in this case \emph{$w$
represents $\succ$ for $\vI$}.
The following lemma shows how we can construct weight vectors representing lexicographical orderings.
Note that each ideal in $\BoolFunc n$ is a principal
ideal\footnote{This follows from the identity $f\cdot(f+g+f\cdot g)=f$ for $f,g\in\BoolFunc n$.}
and each ideal $\langle f\rangle$ in $\BoolFunc n$ corresponds
to an ideal $\langle f,x_{1}^{2}+x_{1},\dots,x_{n}^{2}+x_{n}\rangle$
in $\Fx$. Let us also for simplicity consider the lexicographical ordering defined by $x_{n}\succ x_{n-1}\succ\cdots\succ x_{1}$. The general case can be obtained by permutation.
\begin{lemma}
\label{lem:weight-vector-construction}Consider an ideal of the form $\vI =\langle f,x_{1}^{2}+x_{1},\dots,x_{n}^{2}+x_{n}\rangle\subseteq\Fx$
and the lexicographical ordering $\succ$ defined by $x_{n}\succ x_{n-1}\succ\cdots\succ x_{1}$. Then $\succ$
is represented by the weight vector $w\in\mathbb{Q}^{n}$ with
$w_{k}=1+\sum_{j=1}^{k-1}w_{j}$
or alternatively $w_{k}=2^{k-1}$. 
\end{lemma}

\begin{proof}
Let $\succ$ be as above and $w$ the corresponding weight vector defined there. We first
show that $in_{w}(\vI) \subseteq in_{\succ}(\vI)$. This is true since by definition
of $\vI$ we know that the monomials $x_{1}^{2},\dots,x_{n}^{2}$ are contained
in $in_{w}(\vI)$ (and obviously in $in_{\succ}(\vI)$). Consequently,
we can represent $in_{w}(\vI)$ in the form $in_{w}(\vI)=\langle in_{w}(g_{1}),\dots,in_{w}(g_{k}),x_{1}^{2},\dots,x_{n}^{2}\rangle$
with square free polynomials $g_{1},\dots,g_{k}$ in $\vI$. Now by
construction of $w$ for square free polynomials $g\in\Fx$
the equality $in_{w}(g)=in_{\succ}(g)$ holds. It follows $in_{w}(\vI)\subseteq in_{\succ}(\vI)$.
Analogously it holds $in_{\succ}(\vI)\subseteq in_{w}(\vI)$.
\end{proof}
Let us call a reduced \gb with a distinguished initial
monomial a \emph{marked \gb} in accordance with
\cite[p. 428]{UsingAlgGeom}. Next, we form equivalence classes of
weight vectors which will lead to the same marked \gbs.

\begin{definition}[{\cite[p. 429]{UsingAlgGeom}}]
\label{def:equivalent-orderings}Let $G$ be any marked \gb
for an ideal $\vI$ consisting of $t$ polynomials 
\[
g_{i}=x^{\alpha(i)}+\sum_{\beta}c_{i,\beta}x^{\beta},
\]
where $i\in\{1,\dots,t\}$ and $x^{\alpha(i)}$ is the initial monomial.
We denote with $C_{G}$ the set
\[
C_{G}=\big\{ w\in(\mathbb{R}^{n})^{+}:(\alpha(i)-\beta)\cdot w\geq0\text{ whenever }c_{i,\beta}\not=0\big\}.
\]
\end{definition}

We can combine \lemref{weight-vector-construction}
and \defref{equivalent-orderings} to potentially improve our algorithm.
If we computed for a lexicographical ordering in \algref{compute_min_repr}
a \gb $G$ we can compute and save the equivalence class
$C_{G}$. Now proceeding with the algorithm, for a new lexicographical
ordering we need to check, we can create the corresponding weight
vector $w$ using \lemref{weight-vector-construction} and check if
$w$ is in any of the previously computed equivalence classes $C_{G}$.
If this is the case we can use the result of the previous computation.

Another aspect of the algorithm we can improve is the conversion of \gbs. For an ideal $\vI(B)=\langle f,x_{1}^{2}+x_{1},\dots,x_{n}^{2}+x_{n}\rangle$
 the ring $\Fx\big/\vI(B)$
is zero-dimensional, and
so a finite dimensional vector space. Therefore, it is possible to
use linear algebra for the conversion of \gbs.
This leads to the Faugère-Gianni-Lazard-Mora algorithm (FLGM algorithm)
\cite[p. 49]{faugere1993efficient,UsingAlgGeom}.

\section{Conclusion}
We reformulated \probref{problem-statement} into the language of
algebraic geometry. To do so we described the set of potential classifiers using residue classes modulo the vanishing ideal of the attractors or steady states of the Boolean network. This enabled us to construct an algorithm using normal forms to compute optimal classifiers. Subsequently we demonstrated the usefulness of
this approach by creating an algorithm that produces the minimal solutions
to \probref{problem-statement}. We showed that it is possible to
apply this algorithm to a model for cell-fate decision with $25$
components from \cite{cell-fate-model}. Especially, in combination with reduction algorithms for Boolean networks this allows us to investigate larger networks. 

We hope that it will be also possible to exploit the algebraic reformulation further to speed up computations to tackle even larger networks. Some parts in the algorithm can be improved to obtain potentially
faster computation times. For example the conversion between different \gbs can be done more efficiently using the FLGM algorithm
(see \cite[p. 49-54]{faugere1993efficient,UsingAlgGeom}) which uses
linear algebra for the conversion of \gbs. Since at the moment
of writing this article there was no implementation of this available
in PolyBoRi we did not use this potential improvement. 

However, the main bottleneck for the speed of the algorithm seems
to be the enumeration of possible orderings (or more precisely candidate sets). Therefore, we believe that this will
not lead to a significant increase in speed but this remains to be
tested. Instead we believe, that for larger networks heuristics should be investigated.
Here ideas from the machine learning community could be useful. Potential crosslinks to classifications problems considered there should be explored in the future.

Also different optimality criteria for picking classifiers might
be useful. For example one could try to attribute measurement costs
to components and pick polynomial orderings which lead to optimal
results in such a context as well.

In the introduction we mentioned the relationship of \probref{problem-statement}
to problems in the LAD community. There one starts typically with
a data set of Boolean vectors. Here we focused on the case where our
sets to be classifies are implicitly given. However, approaches developed
for interpolation of Boolean polynomials from data points such as\cite{GrFree}
could be used in the future to tailor our approach to such scenarios
as well.

\bibliographystyle{splncs04}
\bibliography{bibtex}

\newpage

\section*{Appendix}

\begin{table}[h!]

  \caption{\label{tab:classifier-celldeath}$9$ different representations of
classifier for NonACD (see \secref{Biological-example}).}

\begin{tabular}{| p{.30\textwidth} | p{.70\textwidth} |}  \hline
Components & Expression \\
\hline\hline 
$ATP$ & $ATP + 1$ \\
\hline 
$CASP3$, $ Cytc$ & $Cytc + CASP3$ \\
\hline 
$CASP3$, $ MOMP$ & $MOMP + CASP3$ \\
\hline 
$CASP3$, $ SMAC$ & $SMAC + CASP3$ \\
\hline 
$Cytc$, $ apoptosome$ & $Cytc + apoptosome$ \\
\hline 
$MOMP$, $ apoptosome$ & $apoptosome + MOMP$ \\
\hline 
$MPT$ & $MPT$ \\
\hline 
$ROS$ & $ROS$ \\
\hline 
$SMAC$, $ apoptosome$ & $apoptosome + SMAC$ \\
\hline
\end{tabular}

\end{table}

\begin{table}[h!]

  \caption{\label{tab:apoptosis}$17$ different representations of the classifier
for apoptosis (see \secref{Biological-example})}
\begin{tabular}{| p{.30\textwidth} | p{.70\textwidth} |}  \hline
Components & Expression \\
\hline \hline
$ATP, BAX$ & $BAX \cdot ATP$ \\
\hline 
$ATP, CASP8$ & $ATP \cdot CASP8$ \\
\hline 
$ATP, Cytc$ & $Cytc + ATP + 1$ \\
\hline 
$ATP, MOMP$ & $ATP + MOMP + 1$ \\
\hline 
$ATP, SMAC$ & $ATP + SMAC + 1$ \\
\hline 
$BAX, MPT$ & $BAX \cdot MPT + BAX$ \\
\hline 
$BAX, ROS$ & $BAX \cdot ROS + BAX$ \\
\hline 
$CASP3,$ & $CASP3$ \\
\hline 
$CASP8, MPT$ & $MPT \cdot CASP8 + CASP8$ \\
\hline 
$CASP8, ROS$ & $ROS \cdot CASP8 + CASP8$ \\
\hline 
$Cytc, MPT$ & $Cytc + MPT$ \\
\hline 
$Cytc, ROS$ & $Cytc + ROS$ \\
\hline 
$MOMP, MPT$ & $MPT + MOMP$ \\
\hline 
$MOMP, ROS$ & $ROS + MOMP$ \\
\hline 
$MPT, SMAC$ & $MPT + SMAC$ \\
\hline 
$ROS, SMAC$ & $ROS + SMAC$ \\
\hline 
$apoptosome,$ & $apoptosome$ \\
\hline
\end{tabular}

\end{table}

\newpage

\begin{longtable}{| p{.30\textwidth} | p{.70\textwidth} |}
  \caption{84 different representations of the classifier of survival of the cell. (see Section~\ref{sec:Biological-example})} 
\label{tab:survival} \\
  \hline
Components & Expression \\
\hline  \hline
$ATP$, $ BAX$, $ DISCFAS$, $ TNF$ & $DISCFAS \cdot BAX + DISCFAS \cdot ATP + DISCFAS \cdot TNF + BAX \cdot ATP + BAX \cdot TNF + BAX + ATP \cdot TNF$ \\
\hline 
$ATP$, $ BAX$, $ DISCFAS$, $ TNFR$ & $DISCFAS \cdot BAX + DISCFAS \cdot ATP + DISCFAS \cdot TNFR + BAX \cdot ATP + BAX \cdot TNFR + BAX + ATP \cdot TNFR$ \\
\hline 
$ATP$, $ BAX$, $ FADD$, $ FASL$, $ TNF$ & $BAX \cdot FASL \cdot FADD + BAX \cdot TNF + ATP \cdot FASL \cdot FADD \cdot TNF + ATP \cdot FADD \cdot TNF + ATP \cdot TNF + FASL \cdot FADD + FADD \cdot TNF$ \\
\hline 
$ATP$, $ BAX$, $ FADD$, $ FASL$, $ TNFR$ & $BAX \cdot TNFR + BAX \cdot FASL \cdot FADD + ATP \cdot TNFR \cdot FASL \cdot FADD + ATP \cdot TNFR \cdot FADD + ATP \cdot TNFR + TNFR \cdot FADD + FASL \cdot FADD$ \\
\hline 
$ATP$, $ CASP3$, $ DISCFAS$, $ TNF$ & $DISCFAS \cdot ATP \cdot TNF + DISCFAS \cdot ATP + DISCFAS \cdot CASP3 + ATP \cdot TNF + TNF \cdot CASP3$ \\
\hline 
$ATP$, $ CASP3$, $ DISCFAS$, $ TNFR$ & $DISCFAS \cdot ATP \cdot TNFR + DISCFAS \cdot ATP + DISCFAS \cdot CASP3 + ATP \cdot TNFR + TNFR \cdot CASP3$ \\
\hline 
$ATP$, $ CASP3$, $ FADD$, $ FASL$, $ TNF$ & $ATP \cdot FASL \cdot FADD \cdot TNF + ATP \cdot FASL \cdot FADD + ATP \cdot TNF + FASL \cdot FADD \cdot CASP3 + TNF \cdot CASP3$ \\
\hline 
$ATP$, $ CASP3$, $ FADD$, $ FASL$, $ TNFR$ & $ATP \cdot TNFR \cdot FASL \cdot FADD + ATP \cdot TNFR + ATP \cdot FASL \cdot FADD + TNFR \cdot CASP3 + FASL \cdot FADD \cdot CASP3$ \\
\hline 
$ATP$, $ CASP8$, $ DISCFAS$, $ TNF$ & $DISCFAS \cdot TNF \cdot CASP8 + DISCFAS \cdot TNF + DISCFAS \cdot CASP8 + DISCFAS + ATP \cdot TNF \cdot CASP8 + ATP \cdot TNF$ \\
\hline 
$ATP$, $ CASP8$, $ DISCFAS$, $ TNFR$ & $DISCFAS \cdot TNFR \cdot CASP8 + DISCFAS \cdot TNFR + DISCFAS \cdot CASP8 + DISCFAS + ATP \cdot TNFR \cdot CASP8 + ATP \cdot TNFR$ \\
\hline 
$ATP$, $ CASP8$, $ FADD$, $ FASL$, $ TNF$ & $ATP \cdot TNF \cdot CASP8 + ATP \cdot TNF + FASL \cdot FADD \cdot TNF \cdot CASP8 + FASL \cdot FADD \cdot TNF + FASL \cdot FADD \cdot CASP8 + FASL \cdot FADD$ \\
\hline 
$ATP$, $ CASP8$, $ FADD$, $ FASL$, $ TNFR$ & $ATP \cdot TNFR + ATP \cdot FASL \cdot CASP8 + ATP \cdot CASP8 + TNFR \cdot FASL \cdot FADD + TNFR \cdot CASP8 + FASL \cdot FADD \cdot CASP8 + FASL \cdot FADD + FASL \cdot CASP8 + CASP8$ \\
\hline 
$ATP$, $ DISCFAS$, $ TNF$, $ apoptosome$ & $DISCFAS \cdot ATP \cdot TNF + DISCFAS \cdot ATP + DISCFAS \cdot apoptosome + ATP \cdot TNF + apoptosome \cdot TNF$ \\
\hline 
$ATP$, $ DISCFAS$, $ TNFR$, $ apoptosome$ & $DISCFAS \cdot ATP \cdot TNFR + DISCFAS \cdot ATP + DISCFAS \cdot apoptosome + ATP \cdot TNFR + apoptosome \cdot TNFR$ \\
\hline 
$ATP$, $ FADD$, $ FASL$, $ TNF$, $ apoptosome$ & $ATP \cdot FASL \cdot FADD \cdot TNF + ATP \cdot FASL \cdot FADD + ATP \cdot TNF + apoptosome \cdot FASL \cdot FADD + apoptosome \cdot TNF$ \\
\hline 
$ATP$, $ FADD$, $ FASL$, $ TNFR$, $ apoptosome$ & $ATP \cdot TNFR \cdot FASL \cdot FADD + ATP \cdot TNFR + ATP \cdot FASL \cdot FADD + apoptosome \cdot TNFR + apoptosome \cdot FASL \cdot FADD$ \\
\hline 
$ATP$, $ RIP1$ & $ATP \cdot RIP1$ \\
\hline 
$ATP$, $ RIP1K$ & $ATP \cdot RIP1K$ \\
\hline 
$BAX$, $ DISCFAS$, $ MPT$, $ TNF$ & $DISCFAS \cdot BAX + DISCFAS \cdot MPT + DISCFAS \cdot TNF + DISCFAS + BAX \cdot MPT + BAX \cdot TNF + MPT \cdot TNF + TNF$ \\
\hline 
$BAX$, $ DISCFAS$, $ MPT$, $ TNFR$ & $DISCFAS \cdot BAX + DISCFAS \cdot TNFR + DISCFAS \cdot MPT + DISCFAS + BAX \cdot TNFR + BAX \cdot MPT + TNFR \cdot MPT + TNFR$ \\
\hline 
$BAX$, $ DISCFAS$, $ ROS$, $ TNF$ & $DISCFAS \cdot BAX + DISCFAS \cdot ROS + DISCFAS \cdot TNF + DISCFAS + BAX \cdot ROS + BAX \cdot TNF + ROS \cdot TNF + TNF$ \\
\hline 
$BAX$, $ DISCFAS$, $ ROS$, $ TNFR$ & $DISCFAS \cdot BAX + DISCFAS \cdot TNFR + DISCFAS \cdot ROS + DISCFAS + BAX \cdot TNFR + BAX \cdot ROS + TNFR \cdot ROS + TNFR$ \\
\hline 
$BAX$, $ FADD$, $ FASL$, $ MPT$, $ TNF$ & $BAX \cdot FASL \cdot FADD + BAX \cdot TNF + MPT \cdot FASL \cdot FADD \cdot TNF + MPT \cdot FADD \cdot TNF + MPT \cdot TNF + FASL \cdot FADD \cdot TNF + FASL \cdot FADD + TNF$ \\
\hline 
$BAX$, $ FADD$, $ FASL$, $ MPT$, $ TNFR$ & $BAX \cdot TNFR + BAX \cdot FASL \cdot FADD + TNFR \cdot MPT \cdot FASL \cdot FADD + TNFR \cdot MPT \cdot FADD + TNFR \cdot MPT + TNFR \cdot FASL \cdot FADD + TNFR + FASL \cdot FADD$ \\
\hline 
$BAX$, $ FADD$, $ FASL$, $ ROS$, $ TNF$ & $BAX \cdot FASL \cdot FADD + BAX \cdot TNF + FASL \cdot FADD \cdot ROS \cdot TNF + FASL \cdot FADD \cdot TNF + FASL \cdot FADD + FADD \cdot ROS \cdot TNF + ROS \cdot TNF + TNF$ \\
\hline 
$BAX$, $ FADD$, $ FASL$, $ ROS$, $ TNFR$ & $BAX \cdot TNFR + BAX \cdot FASL \cdot FADD + TNFR \cdot FASL \cdot FADD \cdot ROS + TNFR \cdot FASL \cdot FADD + TNFR \cdot FADD \cdot ROS + TNFR \cdot ROS + TNFR + FASL \cdot FADD$ \\
\hline 
$BCL2$ & $BCL2$ \\
\hline 
$CASP3$, $ DISCFAS$, $ MPT$, $ TNF$ & $DISCFAS \cdot MPT \cdot TNF + DISCFAS \cdot MPT + DISCFAS \cdot TNF + DISCFAS \cdot CASP3 + DISCFAS + MPT \cdot TNF + TNF \cdot CASP3 + TNF$ \\
\hline 
$CASP3$, $ DISCFAS$, $ MPT$, $ TNFR$ & $DISCFAS \cdot TNFR \cdot MPT + DISCFAS \cdot TNFR + DISCFAS \cdot MPT + DISCFAS \cdot CASP3 + DISCFAS + TNFR \cdot MPT + TNFR \cdot CASP3 + TNFR$ \\
\hline 
$CASP3$, $ DISCFAS$, $ ROS$, $ TNF$ & $DISCFAS \cdot ROS \cdot TNF + DISCFAS \cdot ROS + DISCFAS \cdot TNF + DISCFAS \cdot CASP3 + DISCFAS + ROS \cdot TNF + TNF \cdot CASP3 + TNF$ \\
\hline 
$CASP3$, $ DISCFAS$, $ ROS$, $ TNFR$ & $DISCFAS \cdot TNFR \cdot ROS + DISCFAS \cdot TNFR + DISCFAS \cdot ROS + DISCFAS \cdot CASP3 + DISCFAS + TNFR \cdot ROS + TNFR \cdot CASP3 + TNFR$ \\
\hline 
$CASP3$, $ FADD$, $ FASL$, $ MPT$, $ TNF$ & $MPT \cdot FASL \cdot FADD \cdot TNF + MPT \cdot FASL \cdot FADD + MPT \cdot TNF + FASL \cdot FADD \cdot TNF + FASL \cdot FADD \cdot CASP3 + FASL \cdot FADD + TNF \cdot CASP3 + TNF$ \\
\hline 
$CASP3$, $ FADD$, $ FASL$, $ MPT$, $ TNFR$ & $TNFR \cdot MPT \cdot FASL \cdot FADD + TNFR \cdot MPT + TNFR \cdot FASL \cdot FADD + TNFR \cdot CASP3 + TNFR + MPT \cdot FASL \cdot FADD + FASL \cdot FADD \cdot CASP3 + FASL \cdot FADD$ \\
\hline 
$CASP3$, $ FADD$, $ FASL$, $ ROS$, $ TNF$ & $FASL \cdot FADD \cdot ROS \cdot TNF + FASL \cdot FADD \cdot ROS + FASL \cdot FADD \cdot TNF + FASL \cdot FADD \cdot CASP3 + FASL \cdot FADD + ROS \cdot TNF + TNF \cdot CASP3 + TNF$ \\
\hline 
$CASP3$, $ FADD$, $ FASL$, $ ROS$, $ TNFR$ & $TNFR \cdot FASL \cdot FADD \cdot ROS + TNFR \cdot FASL \cdot FADD + TNFR \cdot ROS + TNFR \cdot CASP3 + TNFR + FASL \cdot FADD \cdot ROS + FASL \cdot FADD \cdot CASP3 + FASL \cdot FADD$ \\
\hline 
$CASP8$, $ DISCFAS$, $ MPT$, $ TNF$ & $DISCFAS \cdot TNF \cdot CASP8 + DISCFAS \cdot TNF + DISCFAS \cdot CASP8 + DISCFAS + MPT \cdot TNF \cdot CASP8 + MPT \cdot TNF + TNF \cdot CASP8 + TNF$ \\
\hline 
$CASP8$, $ DISCFAS$, $ MPT$, $ TNFR$ & $DISCFAS \cdot TNFR + DISCFAS \cdot MPT + DISCFAS \cdot CASP8 + DISCFAS + TNFR \cdot MPT + TNFR \cdot CASP8 + TNFR + MPT \cdot CASP8$ \\
\hline 
$CASP8$, $ DISCFAS$, $ ROS$, $ TNF$ & $DISCFAS \cdot TNF \cdot CASP8 + DISCFAS \cdot TNF + DISCFAS \cdot CASP8 + DISCFAS + ROS \cdot TNF \cdot CASP8 + ROS \cdot TNF + TNF \cdot CASP8 + TNF$ \\
\hline 
$CASP8$, $ DISCFAS$, $ ROS$, $ TNFR$ & $DISCFAS \cdot TNFR + DISCFAS \cdot ROS + DISCFAS \cdot CASP8 + DISCFAS + TNFR \cdot ROS + TNFR \cdot CASP8 + TNFR + ROS \cdot CASP8$ \\
\hline 
$CASP8$, $ FADD$, $ FASL$, $ MPT$, $ TNF$ & $MPT \cdot TNF \cdot CASP8 + MPT \cdot TNF + FASL \cdot FADD \cdot TNF \cdot CASP8 + FASL \cdot FADD \cdot TNF + FASL \cdot FADD \cdot CASP8 + FASL \cdot FADD + TNF \cdot CASP8 + TNF$ \\
\hline 
$CASP8$, $ FADD$, $ FASL$, $ MPT$, $ TNFR$ & $TNFR \cdot MPT + TNFR \cdot FASL \cdot FADD + TNFR \cdot CASP8 + TNFR + MPT \cdot FASL \cdot CASP8 + MPT \cdot CASP8 + FASL \cdot FADD \cdot CASP8 + FASL \cdot FADD$ \\
\hline 
$CASP8$, $ FADD$, $ FASL$, $ ROS$, $ TNF$ & $FASL \cdot FADD \cdot TNF + FASL \cdot FADD \cdot CASP8 + FASL \cdot FADD + FASL \cdot ROS \cdot CASP8 + ROS \cdot TNF + ROS \cdot CASP8 + TNF \cdot CASP8 + TNF$ \\
\hline 
$CASP8$, $ FADD$, $ FASL$, $ ROS$, $ TNFR$ & $TNFR \cdot FASL \cdot FADD + TNFR \cdot ROS + TNFR \cdot CASP8 + TNFR + FASL \cdot FADD \cdot CASP8 + FASL \cdot FADD + FASL \cdot ROS \cdot CASP8 + ROS \cdot CASP8$ \\
\hline 
$Cytc$, $ DISCFAS$, $ TNF$ & $Cytc \cdot DISCFAS \cdot TNF + Cytc \cdot DISCFAS + Cytc \cdot TNF + DISCFAS \cdot TNF + DISCFAS + TNF$ \\
\hline 
$Cytc$, $ DISCFAS$, $ TNFR$ & $Cytc \cdot DISCFAS \cdot TNFR + Cytc \cdot DISCFAS + Cytc \cdot TNFR + DISCFAS \cdot TNFR + DISCFAS + TNFR$ \\
\hline 
$Cytc$, $ FADD$, $ FASL$, $ TNF$ & $Cytc \cdot FASL \cdot FADD \cdot TNF + Cytc \cdot FASL \cdot FADD + Cytc \cdot TNF + FASL \cdot FADD \cdot TNF + FASL \cdot FADD + TNF$ \\
\hline 
$Cytc$, $ FADD$, $ FASL$, $ TNFR$ & $Cytc \cdot TNFR \cdot FASL \cdot FADD + Cytc \cdot TNFR + Cytc \cdot FASL \cdot FADD + TNFR \cdot FASL \cdot FADD + TNFR + FASL \cdot FADD$ \\
\hline 
$Cytc$, $ RIP1$ & $Cytc \cdot RIP1 + RIP1$ \\
\hline 
$Cytc$, $ RIP1K$ & $Cytc \cdot RIP1K + RIP1K$ \\
\hline 
$DISCFAS$, $ MOMP$, $ TNF$ & $DISCFAS \cdot MOMP \cdot TNF + DISCFAS \cdot MOMP + DISCFAS \cdot TNF + DISCFAS + MOMP \cdot TNF + TNF$ \\
\hline 
$DISCFAS$, $ MOMP$, $ TNFR$ & $DISCFAS \cdot TNFR \cdot MOMP + DISCFAS \cdot TNFR + DISCFAS \cdot MOMP + DISCFAS + TNFR \cdot MOMP + TNFR$ \\
\hline 
$DISCFAS$, $ MPT$, $ TNF$, $ apoptosome$ & $DISCFAS \cdot apoptosome + DISCFAS \cdot MPT \cdot TNF + DISCFAS \cdot MPT + DISCFAS \cdot TNF + DISCFAS + apoptosome \cdot TNF + MPT \cdot TNF + TNF$ \\
\hline 
$DISCFAS$, $ MPT$, $ TNFR$, $ apoptosome$ & $DISCFAS \cdot apoptosome + DISCFAS \cdot TNFR \cdot MPT + DISCFAS \cdot TNFR + DISCFAS \cdot MPT + DISCFAS + apoptosome \cdot TNFR + TNFR \cdot MPT + TNFR$ \\
\hline 
$DISCFAS$, $ ROS$, $ TNF$, $ apoptosome$ & $DISCFAS \cdot apoptosome + DISCFAS \cdot ROS \cdot TNF + DISCFAS \cdot ROS + DISCFAS \cdot TNF + DISCFAS + apoptosome \cdot TNF + ROS \cdot TNF + TNF$ \\
\hline 
$DISCFAS$, $ ROS$, $ TNFR$, $ apoptosome$ & $DISCFAS \cdot apoptosome + DISCFAS \cdot TNFR \cdot ROS + DISCFAS \cdot TNFR + DISCFAS \cdot ROS + DISCFAS + apoptosome \cdot TNFR + TNFR \cdot ROS + TNFR$ \\
\hline 
$DISCFAS$, $ SMAC$, $ TNF$ & $DISCFAS \cdot SMAC \cdot TNF + DISCFAS \cdot SMAC + DISCFAS \cdot TNF + DISCFAS + SMAC \cdot TNF + TNF$ \\
\hline 
$DISCFAS$, $ SMAC$, $ TNFR$ & $DISCFAS \cdot TNFR \cdot SMAC + DISCFAS \cdot TNFR + DISCFAS \cdot SMAC + DISCFAS + TNFR \cdot SMAC + TNFR$ \\
\hline 
$DISCFAS$, $ TNF$, $ cIAP$ & $DISCFAS \cdot TNF \cdot cIAP + DISCFAS \cdot cIAP + TNF \cdot cIAP$ \\
\hline 
$DISCFAS$, $ TNFR$, $ cIAP$ & $DISCFAS \cdot TNFR \cdot cIAP + DISCFAS \cdot cIAP + TNFR \cdot cIAP$ \\
\hline 
$FADD$, $ FASL$, $ MOMP$, $ TNF$ & $FASL \cdot FADD \cdot MOMP \cdot TNF + FASL \cdot FADD \cdot MOMP + FASL \cdot FADD \cdot TNF + FASL \cdot FADD + MOMP \cdot TNF + TNF$ \\
\hline 
$FADD$, $ FASL$, $ MOMP$, $ TNFR$ & $TNFR \cdot FASL \cdot FADD \cdot MOMP + TNFR \cdot FASL \cdot FADD + TNFR \cdot MOMP + TNFR + FASL \cdot FADD \cdot MOMP + FASL \cdot FADD$ \\
\hline 
$FADD$, $ FASL$, $ MPT$, $ TNF$, $ apoptosome$ & $apoptosome \cdot FASL \cdot FADD + apoptosome \cdot TNF + MPT \cdot FASL \cdot FADD \cdot TNF + MPT \cdot FASL \cdot FADD + MPT \cdot TNF + FASL \cdot FADD \cdot TNF + FASL \cdot FADD + TNF$ \\
\hline 
$FADD$, $ FASL$, $ MPT$, $ TNFR$, $ apoptosome$ & $apoptosome \cdot TNFR + apoptosome \cdot FASL \cdot FADD + TNFR \cdot MPT \cdot FASL \cdot FADD + TNFR \cdot MPT + TNFR \cdot FASL \cdot FADD + TNFR + MPT \cdot FASL \cdot FADD + FASL \cdot FADD$ \\
\hline 
$FADD$, $ FASL$, $ ROS$, $ TNF$, $ apoptosome$ & $apoptosome \cdot FASL \cdot FADD + apoptosome \cdot TNF + FASL \cdot FADD \cdot ROS \cdot TNF + FASL \cdot FADD \cdot ROS + FASL \cdot FADD \cdot TNF + FASL \cdot FADD + ROS \cdot TNF + TNF$ \\
\hline 
$FADD$, $ FASL$, $ ROS$, $ TNFR$, $ apoptosome$ & $apoptosome \cdot TNFR + apoptosome \cdot FASL \cdot FADD + TNFR \cdot FASL \cdot FADD \cdot ROS + TNFR \cdot FASL \cdot FADD + TNFR \cdot ROS + TNFR + FASL \cdot FADD \cdot ROS + FASL \cdot FADD$ \\
\hline 
$FADD$, $ FASL$, $ SMAC$, $ TNF$ & $FASL \cdot FADD \cdot SMAC \cdot TNF + FASL \cdot FADD \cdot SMAC + FASL \cdot FADD \cdot TNF + FASL \cdot FADD + SMAC \cdot TNF + TNF$ \\
\hline 
$FADD$, $ FASL$, $ SMAC$, $ TNFR$ & $TNFR \cdot FASL \cdot FADD \cdot SMAC + TNFR \cdot FASL \cdot FADD + TNFR \cdot SMAC + TNFR + FASL \cdot FADD \cdot SMAC + FASL \cdot FADD$ \\
\hline 
$FADD$, $ FASL$, $ TNF$, $ cIAP$ & $FASL \cdot FADD \cdot TNF \cdot cIAP + FASL \cdot FADD \cdot cIAP + TNF \cdot cIAP$ \\
\hline 
$FADD$, $ FASL$, $ TNFR$, $ cIAP$ & $TNFR \cdot FASL \cdot FADD \cdot cIAP + TNFR \cdot cIAP + FASL \cdot FADD \cdot cIAP$ \\
\hline 
$IKK$ & $IKK$ \\
\hline 
$MOMP$, $ RIP1$ & $MOMP \cdot RIP1 + RIP1$ \\
\hline 
$MOMP$, $ RIP1K$ & $RIP1K \cdot MOMP + RIP1K$ \\
\hline 
$MPT$, $ RIP1$ & $MPT \cdot RIP1 + RIP1$ \\
\hline 
$MPT$, $ RIP1K$ & $RIP1K \cdot MPT + RIP1K$ \\
\hline 
$NFKB1$ & $NFKB1$ \\
\hline 
$RIP1$, $ ROS$ & $ROS \cdot RIP1 + RIP1$ \\
\hline 
$RIP1$, $ SMAC$ & $SMAC \cdot RIP1 + RIP1$ \\
\hline 
$RIP1$, $ cIAP$ & $RIP1 \cdot cIAP$ \\
\hline 
$RIP1K$, $ ROS$ & $RIP1K \cdot ROS + RIP1K$ \\
\hline 
$RIP1K$, $ SMAC$ & $RIP1K \cdot SMAC + RIP1K$ \\
\hline 
$RIP1K$, $ cIAP$ & $RIP1K \cdot cIAP$ \\
\hline 
$RIP1ub$ & $RIP1ub$ \\
\hline 
$XIAP$ & $XIAP$ \\
\hline 
$cFLIP$ & $cFLIP$ \\
\hline
\end{longtable}

\end{document}